\documentclass[reqno]{amsart}
\usepackage{amssymb,latexsym}
\usepackage{graphicx}
\usepackage{fancyhdr}
\numberwithin{equation}{section}
\newtheorem{thm}{Theorem}[section]
\newtheorem{theorem}[thm]{Theorem}

\newtheorem{corollary}[thm]{Corollary}

\begin{document}

\setcounter{page}{1}

\title[A unification of summation formulas]{A unification of Euler-Maclaurin and Euler-Boole summation formulas }
\author{Yuan He}
\address{School of Mathematics and Big Data, Neijiang Normal University, Neijiang 641100, Sichuan, People's Republic of China}
\email{hyyhe@aliyun.com}

\begin{abstract}
In this paper, we establish a summation formula associated with the generalized Apostol-Bernoulli functions. This formula unifies the Euler-Maclaurin summation formula, the Euler-Boole summation formula, and the character analogues of these two formulas. We also apply it to obtain the general power sum formula and the special values of Berndt's generalized $L$-function at non-positive integers.
\end{abstract}

\maketitle

\section{Introduction}\label{sec1}

Throughout this paper, let $\mathbb{N}$ be the set of positive integers, $\mathbb{N}_{0}$ the set of non-negative integers, $\mathbb{Z}$ the set of integers, $\mathbb{R}$ the set of real numbers, and $\mathbb{C}$ the set of complex numbers. It is a fundamental fact that the summation and approximation of infinite series constitute central problems in analysis. Since the 18th century, mathematicians have sought universal methods for determining the sums of convergent series with increasing rigor. In this context, Euler \cite{euler1} and Maclaurin \cite{maclaurin} independently developed what is now known as the Euler-Maclaurin summation formula:
\begin{eqnarray}\label{eq1.1}
\sum_{k=1}^{n}f(k)&=&\int_{0}^{n}f(t)dt+\frac{f(n)-f(0)}{2}\nonumber\\
&&+\sum_{k=1}^{\infty}\frac{B_{2k}}{(2k)!}\bigl(f^{(2k-1)}(n)-f^{(2k-1)}(0)\bigl),
\end{eqnarray}
where $f$ is a nice function, $f^{(m)}(x)$ denotes the $m$-th derivative of $f(x)$ with respect to $x$, and $B_{n}$ are the Bernoulli numbers (as they appear in \eqref{eq1.3} below with $B_{n}=B_{n}(0)$, given by Euler \cite{euler4}).
Since $|B_{2k}|$ grows rather rapidly (in fact, super-exponentially; see \cite[p. 267]{apostol3} for details), the series on the right side of \eqref{eq1.1} will seldom converge. This means the exact theory of \eqref{eq1.1} belongs more to the theory of asymptotic series than to that of summable series. Nevertheless, it is very important in many branches of analysis; see Varadarajan \cite{varadarajan} for comments on Euler's use of \eqref{eq1.1}.

To use \eqref{eq1.1} effectively one must have remainder terms, and this was carried out by Euler's successors. Poisson \cite{poisson} was the first to seriously discuss remainder terms, and the first fully rigorous treatment was given by Jacobi \cite{jacobi}. Furthermore, Hardy, in his book \cite[Chapter 13]{hardy2}, established the summability of the Euler-Maclaurin formula under some fairly general conditions. Cohen, in his book \cite[Theorem 9.2.2]{cohen}, wrote its general form as:
\begin{eqnarray}\label{eq1.2}
\sum_{\substack{a<k\leqslant b\\k\in\mathbb{Z}}}f(k)&=&\int_{a}^{b}f(t)dt+\sum_{k=0}^{m}\frac{(-1)^{k+1}}{(k+1)!}\bigl(B_{k+1}(\{b\})f^{(k)}(b)-B_{k+1}(\{a\})f^{(k)}(a)\bigl)\nonumber\\
&&+\frac{(-1)^{m}}{(m+1)!}\int_{a}^{b}B_{m+1}(\{t\})f^{(m+1)}(t)dt,
\end{eqnarray}
where $m\in\mathbb{N}_{0}$, $a,b\in\mathbb{R}$ with $a<b$, $f(t)$ is $m+1$
times continuously differentiable on the interval $[a,b]$, $B_{n}(\{x\})$ are the periodic extensions of the Bernoulli polynomials of degree $n$ with $\{x\}$ being the fractional part of $x\in\mathbb{R}$, and the Bernoulli polynomials of degree $n$ are defined by the generating function
\begin{equation}\label{eq1.3}
\frac{te^{xt}}{e^{t}-1}=\sum_{n=0}^{\infty}B_{n}(x)\frac{t^{n}}{n!}\quad(|t|<2\pi).
\end{equation}
An immediate application of \eqref{eq1.2} with $f(t)=t^{m}$, $a=0$ and $b=n$ gives the power sum formula (see, e.g., \cite[pp. 433--434]{euler4} or \cite[p. 230]{ireland}):
\begin{equation}\label{eq1.4}
\sum_{k=1}^{n}k^{m}=\frac{1}{m+1}\sum_{k=0}^{m}\binom{m+1}{k}(-1)^{k}B_{k}n^{m+1-k},
\end{equation}
where $n\in\mathbb{N}$ and $m\in\mathbb{N}_{0}$. Another direct application of \eqref{eq1.2} with $f(t)=t^{-1}$, $a=1$, $b=n$ and $m=0$ yields
the so-called Euler constant (see, e.g., \cite{euler2,knuth}):
\begin{equation}\label{eq1.5}
\gamma=\underset{n\rightarrow \infty}{\lim}\biggl(\sum_{k=1}^{n}\frac{1}{k}-\log n\biggl)=0.57721566490\cdots.
\end{equation}
In addition, taking $f(t)=\log t$, $a=1$, $b=n$ and setting $m=1$ in \eqref{eq1.2} recovers the formula of Stirling \cite{stirling}, known as Stirling's formula:
\begin{equation}\label{eq1.6}
n!\sim \sqrt{2\pi n}\biggl(\frac{n}{e}\biggl)^{n},
\end{equation}
where $\sim$ denotes asymptotic equivalence as $n\rightarrow\infty$. For the complex form of \eqref{eq1.6}, the reader is referred to \cite[Proposition 9.6.27]{cohen}; this can be regarded as a further application of \eqref{eq1.2}. We also remark that Titchmarsh, in his book \cite[pp. 13--15]{titchmarsh}, showed that the case $m=0$ in \eqref{eq1.2} can be used to prove Riemann's functional equation \cite{riemann} for the Riemann zeta function. Moreover, Berndt \cite{berndt1,berndt2} used \eqref{eq1.2}, with $m=0$, to prove Hurwitz's formula \cite{hurwitz} for the Hurwitz zeta-function and Lerch's functional equation \cite{lerch} for the Lerch zeta-function. Later, in 1975, Berndt \cite[Theorem 4.1]{berndt3} established the character analogue of \eqref{eq1.2}: Let $q\in\mathbb{N}$, $m\in\mathbb{N}_{0}$ and $a,b\in\mathbb{R}$ with $q\geqslant2$ and $a<b$. Let $\chi$ be a primitive Dirichlet character modulo $q$ and let $f(t)$ be an $m+1$
times continuously differentiable function on the interval $[a,b]$. Then
\begin{eqnarray}\label{eq1.7}
\sideset{}{'}\sum_{\substack{a\leqslant k\leqslant b\\k\in\mathbb{Z}}}\chi(k)f(k)&=&\chi(-1)\sum_{k=0}^{m}\frac{(-1)^{k+1}}{(k+1)!}\bigl(\overline{B}_{k+1}(b,\chi)f^{(k)}(b)-\overline{B}_{k+1}(a,\chi)f^{(k)}(a)\bigl)\nonumber\\
&&+\chi(-1)\frac{(-1)^{m}}{(m+1)!}\int_{a}^{b}\overline{B}_{m+1}(t,\chi)f^{(m+1)}(t)dt,
\end{eqnarray}
where the dash indicates that if $n=a$ or $n=b$, then only $\frac{1}{2}\chi(a)f(a)$ or $\frac{1}{2}\chi(b)f(b)$, respectively, is
counted, and $\overline{B}_{n}(x,\chi)$ are the generalized Bernoulli functions given by
\begin{equation*}
\overline{B}_{n}(x,\chi)=q^{n-1}\sum_{k=1}^{q}\chi(k)\overline{B}_{n}\biggl(\frac{x+k}{q}\biggl),
\end{equation*}
in which $\overline{B}_{n}(x)$ is the $n$-th Bernoulli function, defined for $n\in\mathbb{N}$ and $x\in\mathbb{R}$ as follows: $\overline{B}_{1}(x)=0$ if $x\in\mathbb{Z}$, and $\overline{B}_{1}(x)=B_{1}(\{x\})$ if $x\in\mathbb{R}\setminus\mathbb{Z}$; for $n\geqslant2$, $\overline{B}_{n}(x)=B_{n}(\{x\})$. Meanwhile, Berndt \cite[Theorem 5.1]{berndt3} used \eqref{eq1.7}, with $m=0$, to find his functional equation for Berndt's generalized $L$-function. In 2015, Da\u{g}l{\i} and Can \cite[Theorems 1 and 2]{dagli} used \eqref{eq1.7} to establish some reciprocity formulas for Dedekind character sums involving two primitive characters.

In fact, following the proof of \eqref{eq1.2} given by Cohen \cite[Theorem 9.2.2]{cohen}, we can obtain its alternating version, namely the Euler-Boole summation formula: Let $m\in\mathbb{N}_{0}$ and $a,b\in\mathbb{R}$ with $a<b$. Let $f(t)$ be an $m+1$ times continuously differentiable function on the interval $[a,b]$. Then
\begin{eqnarray}\label{eq1.8}
\sum_{\substack{a<k\leqslant b\\k\in\mathbb{Z}}}(-1)^{k}f(k)&=&\frac{1}{2}\sum_{k=0}^{m}\frac{(-1)^{k}}{k!}\bigl(\widetilde{E}_{k}(b)f^{(k)}(b)-\widetilde{E}_{k}(a)f^{(k)}(a)\bigl)\nonumber\\
&&+\frac{(-1)^{m+1}}{2m!}\int_{a}^{b}\widetilde{E}_{m}(\{t\})f^{(m+1)}(t)d t,
\end{eqnarray}
where $\widetilde{E}_{n}(x)=(-1)^{[x]}E_{n}(\{x\})$ are the periodic extensions of the Euler polynomials of degree $n$ (with $[x]$ the integer part of $x\in\mathbb{R}$ satisfying $\{x\}+[x]=x$), and the Euler polynomials of degree $n$ are defined by the generating function
\begin{equation}\label{eq1.9}
\frac{2e^{xt}}{e^{t}+1}=\sum_{n=0}^{\infty}E_{n}(x)\frac{t^{n}}{n!}\quad(|t|<\pi).
\end{equation}
It should be noted that Euler \cite{euler3} was the first to discuss \eqref{eq1.8} without the remainder terms, while the first to present \eqref{eq1.8} with the remainder terms was Boole \cite{boole}. In particular, Euler \cite{euler5} applied his summation formula to a nice function $f$,
\begin{equation}\label{eq1.10}
\sum_{k=0}^{\infty}(-1)^{k}f(k)=\frac{1}{2}f(0)+\sum_{k=1}^{\infty}\frac{(1-2^{2k})B_{2k}}{(2k)!}f^{(2k-1)}(0),
\end{equation}
in order to find the functional equations for two alternating series. Helpful summaries of Euler's mathematical ideas from this paper are found in Hardy \cite[Chapter 2]{hardy2} and Ayoub \cite{ayoub}. It is worth mentioning that Hu et al. \cite{hu}, in 2015, used the special case $a,b\in\mathbb{Z}$ in \eqref{eq1.8} to establish two reciprocity formulas for arithmetic sums similar to the generalized Dedekind sums considered by Apostol \cite{apostol1}. More recently, in 2017, Can and Da\u{g}l{\i} \cite[Theorem 1.3]{can} used \eqref{eq1.7} to derive the character analogue of \eqref{eq1.8} as follows:
\begin{eqnarray}\label{eq1.11}
\sum_{\substack{a<k< b\\k\in\mathbb{Z}}}(-1)^{k}\chi(k)f(k)&=&\frac{\chi(-1)}{2}\sum_{k=0}^{m}\frac{(-1)^{k}}{k!}\bigl(\widetilde{E}_{k}(b,\chi)f^{(k)}(b)-\widetilde{E}_{k}(a,\chi)f^{(k)}(a)\bigl)\nonumber\\
&&+\frac{(-1)^{m+1}\chi(-1)}{2m!}\int_{a}^{b}\widetilde{E}_{m}(t,\chi)f^{(m+1)}(t)d t,
\end{eqnarray}
where $q\in\mathbb{N}$ satisfies $q\geqslant2$ and $2\nmid q$, $m\in\mathbb{N}_{0}$, $a,b\in\mathbb{R}$ satisfy $a<b$, $\chi$ is a primitive Dirichlet character modulo $q$, $f(t)$ is $m+1$
times continuously differentiable on the interval $[a,b]$, and $\widetilde{E}_{n}(x,\chi)$ are given by
\begin{equation*}
\widetilde{E}_{n}(x,\chi)=q^{n}\sum_{k=1}^{q}(-1)^{k}\chi(k)\widetilde{E}_{n}\biggl(\frac{x+k}{q}\biggl).
\end{equation*}
As applications of \eqref{eq1.11}, Can and Da\u{g}l{\i} \cite[Theorems 1.5 and 1.6]{can} used it to establish some reciprocity formulas for arithmetic sums similar to the Hardy-Berndt sums considered by Hardy \cite{hardy1} and Berndt \cite{berndt4}.

Naturally, one may ask whether there exists a more general summation formula that unifies the above results \eqref{eq1.2}, \eqref{eq1.7}, \eqref{eq1.8} and \eqref{eq1.11}. We here give a positive answer. Building on the author's results \cite{he} on the generalized Apostol-Bernoulli polynomials and functions, we establish a summation formula associated with the generalized Apostol-Bernoulli functions (see Theorem \ref{thm2.1} below).

The remainder of this paper is organized as follows. In Section \ref{sec2}, we present our main result and show that it subsumes \eqref{eq1.2}, \eqref{eq1.7}, \eqref{eq1.8} and \eqref{eq1.11} as special cases; we also apply it to obtain the general power sum formula and the special values of Berndt's generalized $L$-function at non-positive integers. In Section \ref{sec3}, we give the detailed proof of Theorem \ref{thm2.1}.

\section{Main result and its applications}\label{sec2}

In what follows we define $\chi^{-}$ by $\chi^{-}(n)=\chi(-n)$ for a Dirichlet character $\chi$ modulo $q\in\mathbb{N}$ and $n\in\mathbb{Z}$. We also define, for $z\in\mathbb{C}$ and $\lambda\in\mathbb{C}\setminus\{0\}$, $\lambda^{z}=e^{z\log\lambda}$ as the general power function, with the principal branch of the logarithm satisfying
\begin{equation*}
\log \lambda=\log|\lambda|+\mathrm{i}\arg \lambda \quad( -\pi<\arg \lambda\leqslant\pi),
\end{equation*}
and let $\delta_{s,z}$ be the indicator function defined by $\delta_{s,z}=1$ if $s=z$, and $0$ otherwise, for $s,z\in\mathbb{C}$. We now state our main result as follows.

\begin{theorem}\label{thm2.1} Let $q\in\mathbb{N}$, $m\in\mathbb{N}_{0}$ and $a,b\in\mathbb{R}$ with $a<b$. Let $\chi$ be a Dirichlet character modulo $q$ and let $f(t)$ be an $m+1$
times continuously differentiable function on the interval $[a,b]$. Then, for $\lambda\in\mathbb{C}\setminus\{0\}$,
\begin{eqnarray}\label{eq2.1}
&&\sum_{\substack{a<k\leqslant b\\k\in\mathbb{Z}}}\chi(k)\lambda^{-k}f(k)\nonumber\\
&&\qquad=\sum_{k=0}^{m}\frac{(-1)^{k+1}}{(k+1)!}\bigl(\beta_{k+1,\chi^{-}}(\{b\}_{\chi},\lambda)f^{(k)}(b)-\beta_{k+1,\chi^{-}}(\{a\}_{\chi},\lambda)f^{(k)}(a)\bigl)\nonumber\\
&&\qquad\quad+\delta_{1,\lambda^{q}}\beta_{0,\chi^{-}}(\lambda)\int_{a}^{b}f(t)dt\nonumber\\
&&\qquad\quad+\frac{(-1)^{m}}{(m+1)!}\int_{a}^{b}\beta_{m+1,\chi^{-}}(\{t\}_{\chi},\lambda)f^{(m+1)}(t)dt,
\end{eqnarray}
where
\begin{equation*}
\beta_{0,\chi^{-}}(\lambda)=\frac{1}{q}\sum_{r=1}^{q}\chi(-r)\lambda^{r},
\end{equation*}
and
\begin{equation*}
\beta_{m,\chi^{-}}(\{x\}_{\chi},\lambda)=q^{m-1}\sum_{r=1}^{q}\chi(-r)\lambda^{r-q[\frac{x+r}{q}]}\beta_{m}\biggl(\biggl\{\frac{x+r}{q}\biggl\},\lambda^{q}\biggl),
\end{equation*}
in which $\beta_{n}(x,\lambda)$ denotes the Apostol-Bernoulli polynomials defined for $\lambda\in\mathbb{C}\setminus\{0\}$ by the generating function (see, e.g., \cite{apostol2})
\begin{equation}\label{eq2.2}
\frac{te^{xt}}{\lambda e^{t}-1}=\sum_{n=0}^{\infty}
\beta_{n}(x,\lambda)\frac{t^{n}}{n!}\quad(|t+\log\lambda|<2\pi).
\end{equation}
\end{theorem}

It follows that we show some special cases of Theorem \ref{thm2.1}. We first give the following result.

\begin{corollary}\label{cor2.2} Let $q\in\mathbb{N}$, $m\in\mathbb{N}_{0}$ and $a,b\in\mathbb{R}$ with $a<b$. Let $\chi$ be a Dirichlet character modulo $q$ and let $f(t)$ be an $m+1$
times continuously differentiable function on the interval $[a,b]$. Then
\begin{eqnarray}\label{eq2.3}
&&\sum_{\substack{a<k\leqslant b\\k\in\mathbb{Z}}}\chi(k)f(k)\nonumber\\
&&\qquad=\sum_{k=0}^{m}\frac{(-1)^{k+1}}{(k+1)!}\bigl(B_{k+1}(\{b\}_{\chi},\chi^{-})f^{(k)}(b)-B_{k+1}(\{a\}_{\chi},\chi^{-})f^{(k)}(a)\bigl)\nonumber\\
&&\qquad\quad+B_{0}(\chi^{-})\int_{a}^{b}f(t)dt\nonumber\\
&&\qquad\quad+\frac{(-1)^{m}}{(m+1)!}\int_{a}^{b}B_{m+1}(\{t\}_{\chi},\chi^{-})f^{(m+1)}(t)dt,
\end{eqnarray}
where
\begin{equation*}
B_{0}(\chi^{-})=\frac{1}{q}\sum_{r=1}^{q}\chi(-r),
\end{equation*}
and
\begin{equation*}
B_{m}(\{x\}_{\chi},\chi^{-})=q^{m-1}\sum_{r=1}^{q}\chi(-r)B_{m}\biggl(\biggl\{\frac{x+r}{q}\biggl\}\biggl).
\end{equation*}
\end{corollary}

\begin{proof}
Setting $\lambda=1$ in Theorem \ref{thm2.1} gives the desired result.
\end{proof}

Clearly, the case $q=1$ in Corollary \ref{cor2.2} gives \eqref{eq1.2}. If we take $q\geqslant2$ and let $\chi$ be a primitive Dirichlet character modulo $q$, then by Corollary \ref{cor2.2} we recover \eqref{eq1.7}. We here mention that an equivalent version of Corollary \ref{cor2.2} appears in \cite[Proposition 9.4.15]{cohen}.

We also have the following alternating versions of Corollary \ref{cor2.2}.

\begin{corollary}\label{cor2.3} Let $q\in\mathbb{N}$, $m\in\mathbb{N}_{0}$ and $a,b\in\mathbb{R}$ with $2\nmid q$ and $a<b$. Let $\chi$ be a Dirichlet character modulo $q$ and let $f(t)$ be an $m+1$
times continuously differentiable function on the interval $[a,b]$. Then
\begin{eqnarray}\label{eq2.4}
&&\sum_{\substack{a<k\leqslant b\\k\in\mathbb{Z}}}(-1)^{k}\chi(k)f(k)\nonumber\\
&&\qquad=\frac{1}{2}\sum_{k=0}^{m}\frac{(-1)^{k}}{k!}\bigl(E_{k}(\{b\}_{\chi},\chi^{-})f^{(k)}(b)-E_{k}(\{a\}_{\chi},\chi^{-})f^{(k)}(a)\bigl)\nonumber\\
&&\qquad\quad+\frac{(-1)^{m+1}}{2m!}\int_{a}^{b}E_{m}(\{t\}_{\chi},\chi^{-})f^{(m+1)}(t)dt,
\end{eqnarray}
where
\begin{equation*}
E_{m}(\{x\}_{\chi},\chi^{-})=q^{m}\sum_{r=1}^{q}(-1)^{r-[\frac{x+r}{q}]}\chi(-r)E_{m}\biggl(\biggl\{\frac{x+r}{q}\biggl\}\biggl).
\end{equation*}
\end{corollary}

\begin{proof}
From \eqref{eq1.9} and \eqref{eq2.2}, it follows that for $n\in\mathbb{N}_{0}$,
\begin{equation}\label{eq2.5}
E_{n}(x)=-\frac{2\beta_{n+1}(x,-1)}{n+1}.
\end{equation}
Thus, by setting $\lambda=-1$ in Theorem \ref{thm2.1} and in light of \eqref{eq2.5}, the desired result follows.
\end{proof}

\begin{corollary}\label{cor2.4} Let $q\in\mathbb{N}$, $m\in\mathbb{N}_{0}$ and $a,b\in\mathbb{R}$ with $2\mid q$ and $a<b$. Let $\chi$ be a Dirichlet character modulo $q$ and let $f(t)$ be an $m+1$
times continuously differentiable function on the interval $[a,b]$. Then
\begin{eqnarray}\label{eq2.6}
&&\sum_{\substack{a<k\leqslant b\\k\in\mathbb{Z}}}(-1)^{k}\chi(k)f(k)\nonumber\\
&&\qquad=\sum_{k=0}^{m}\frac{(-1)^{k+1}}{(k+1)!}\bigl(\widehat{B}_{k+1}(\{b\}_{\chi},\chi^{-})f^{(k)}(b)-\widehat{B}_{k+1,\chi^{-}}(\{a\}_{\chi},\chi^{-})f^{(k)}(a)\bigl)\nonumber\\
&&\qquad\quad+\widehat{B}_{0}(\chi^{-})\int_{a}^{b}f(t)dt\nonumber\\
&&\qquad\quad+\frac{(-1)^{m}}{(m+1)!}\int_{a}^{b}\widehat{B}_{m+1}(\{t\}_{\chi},\chi^{-})f^{(m+1)}(t)dt,
\end{eqnarray}
where
\begin{equation*}
\widehat{B}_{0}(\chi^{-})=\frac{1}{q}\sum_{r=1}^{q}(-1)^{r}\chi(-r),
\end{equation*}
and
\begin{equation*}
\widehat{B}_{m}(\{x\}_{\chi},\chi^{-})=q^{m-1}\sum_{r=1}^{q}(-1)^{r}\chi(-r)B_{m}\biggl(\biggl\{\frac{x+r}{q}\biggl\}\biggl).
\end{equation*}
\end{corollary}

\begin{proof}
Setting $\lambda=-1$ in Theorem \ref{thm2.1} gives the desired result.
\end{proof}

It is easily seen that Corollary \ref{cor2.3} includes all the information of \eqref{eq1.8} and \eqref{eq1.11}, and Corollary \ref{cor2.4} is a natural complement to Corollary \ref{cor2.3}.

The following result is a more general statement than the power
sum formula \eqref{eq1.4}.

\begin{theorem}\label{thm2.5} Let $n,q\in\mathbb{N}$ and $m\in\mathbb{N}_{0}$. Let $\chi$ be a Dirichlet character modulo $q$. Then, for $\lambda\in\mathbb{C}\setminus\{0\}$,
\begin{eqnarray}\label{eq2.7}
\sum_{k=1}^{nq}\chi(k)\lambda^{-k}k^{m}&=&\frac{\lambda^{-qn}}{m+1}\sum_{k=0}^{m}\binom{m+1}{k}(-1)^{k}\beta_{k,\chi^{-}}(\{0\}_{\chi},\lambda)(nq)^{m+1-k}\nonumber\\
&&+\frac{(-1)^{m+1}\beta_{m+1,\chi^{-}}(\{0\}_{\chi},\lambda)(\lambda^{-nq}-1)}{m+1}.
\end{eqnarray}
\end{theorem}

\begin{proof}
Since
\begin{equation*}
\beta_{m,\chi^{-}}(\{nq\}_{\chi},\lambda)=\lambda^{-nq}\beta_{m,\chi^{-}}(\{0\}_{\chi},\lambda),
\end{equation*}
by applying Theorem \ref{thm2.1} with $f(t)=t^{m}$, $a=0$ and $b=nq$, we have
\begin{eqnarray}\label{eq2.8}
\sum_{k=1}^{nq}\chi(k)\lambda^{-k}k^{m}&=&\frac{\lambda^{-nq}}{m+1}\sum_{k=1}^{m}\binom{m+1}{k}(-1)^{k}\beta_{k,\chi^{-}}(\{0\}_{\chi},\lambda)(nq)^{m+1-k}\nonumber\\
&&+\frac{(-1)^{m+1}\beta_{m+1,\chi^{-}}(\{0\}_{\chi},\lambda)(\lambda^{-nq}-1)}{m+1}\nonumber\\
&&+\frac{\delta_{1,\lambda^{q}}\beta_{0,\chi^{-}}(\lambda)(nq)^{m+1}}{m+1}.
\end{eqnarray}
Note that by \eqref{eq2.2} we have
\begin{equation}\label{eq2.9}
(\lambda e^{t}-1)\sum_{n=0}^{\infty}\beta_{n}(x,\lambda)\frac{t^{n}}{n!}=te^{xt}.
\end{equation}
Hence, applying the Taylor series of $e^{t}$ to both sides of \eqref{eq2.9}, we find that
\begin{equation}\label{eq2.10}
\beta_{0}(x,\lambda)=\begin{cases}
1,  &\lambda=1,\\
0,  &\lambda\not=1,
\end{cases}
\end{equation}
and
\begin{equation}\label{eq2.11}
\beta_{1}(x,\lambda)=\begin{cases}
x-\frac{1}{2},  &\lambda=1,\\
\frac{1}{\lambda-1},  &\lambda\not=1.
\end{cases}
\end{equation}
Thus, by applying \eqref{eq2.10} to the right-hand side of \eqref{eq2.8}, we prove Theorem \ref{thm2.5}.
\end{proof}

Obviously, the case $q=\lambda=1$ in Theorem \ref{thm2.5} immediately gives \eqref{eq1.4}. If we take $q=1$ and $\lambda=-1$ in Theorem \ref{thm2.5}, then in view of \eqref{eq2.5} and \eqref{eq2.10} we obtain the alternating power sum formula:
\begin{equation}\label{eq2.12}
\sum_{k=1}^{n}(-1)^{k}k^{m}=\frac{(-1)^{n}}{2}\sum_{k=0}^{m}\binom{m}{k}(-1)^{k}E_{k}(0)n^{m-k}+\frac{(-1)^{m+1}E_{m}(0)}{2},
\end{equation}
or equivalently,
\begin{eqnarray}\label{eq2.13}
\sum_{k=1}^{n}(-1)^{k}k^{m}&=&(-1)^{n}\sum_{k=0}^{m}\binom{m}{k}(-1)^{k}\frac{(1-2^{k+1})B_{k+1}n^{m-k}}{k+1}\nonumber\\
&&+\frac{(-1)^{m+1}(1-2^{m+1})B_{m+1}}{m+1},
\end{eqnarray}
where $n\in\mathbb{N}$ and $m\in\mathbb{N}_{0}$. It should be noted that the first few values  of formula \eqref{eq2.13} also appears in Euler's book \cite[pp. 499--501]{euler4}.

We now turn our attention to Berndt's generalized $L$-function, which is defined by the series (see, e.g., \cite[p. 430]{berndt3})
\begin{equation}\label{eq2.14}
L(s,x,a,\chi)=\sideset{}{'}\sum_{n=0}^{\infty}\frac{\chi(n)e^{\frac{2\pi \mathrm{i}nx}{q}}}{(n+a)^{s}},
\end{equation}
where $x\in\mathbb{R}$, $a\in\mathbb{C}$, $\chi$ is a Dirichlet character $\chi$ modulo $q\in\mathbb{N}$, the dash indicates that the term corresponding to $n=-a$ is omitted if $a\in\{0,-1,-2,\ldots\}$, $\Re(s)>1$ if $x\in\mathbb{Z}$ and $\gcd(x,q)=1$, and $\Re(s)>0$ otherwise. It is evident that \eqref{eq2.14} includes Dirichlet $L$-function, Riemann zeta-function, Hurwitz-zeta function, alternating Hurwitz-zeta function, and Lerch-zeta function as special cases.

We next apply Theorem \ref{thm2.1} to obtain the special values of Berndt's generalized $L$-function at non-positive integers.

\begin{theorem}\label{thm2.6} Let $n,q\in\mathbb{N}$ and $x,a\in\mathbb{R}$ with $a>0$. Let $\chi$ be a Dirichlet character modulo $q$. If $q=1$, then
\begin{equation}\label{eq2.15}
L(1-n,x,a,\chi)=-\frac{\beta_{n}(a,e^{\frac{2\pi \mathrm{i}x}{q}})}{n},
\end{equation}
and if $q\geqslant2$, then
\begin{equation}\label{eq2.16}
L(1-n,x,a,\chi)=-\frac{\beta_{n,\chi}(a,e^{\frac{2\pi \mathrm{i}x}{q}})}{n},
\end{equation}
where $\beta_{n,\chi}(x,\lambda)$ are the generalized Apostol-Bernoulli polynomials, defined for $\lambda\in\mathbb{C}\setminus\{0\}$ and for a Dirichlet character $\chi$ modulo $q\in\mathbb{N}$, by the generating function (see, e.g., \cite[Definition 2.1]{he})
\begin{equation}\label{eq2.17}
\sum_{r=1}^{q}\frac{\chi(r)\lambda^{r}te^{(r+x)t}}{\lambda^{q}e^{qt}-1}=\sum_{n=0}^{\infty}\beta_{n,\chi}(x,\lambda)\frac{t^{n}}{n!}\quad(|t+\log\lambda|<2\pi/q).
\end{equation}
\end{theorem}

\begin{proof} Taking $a=0$, $b=n$ and $f(t)=(t+y)^{s}$, and replacing $m$ by $m-1$ in Theorem \ref{thm2.1}, we find that for $m\in\mathbb{N}$, $n\in\mathbb{N}_{0}$, $y\in\mathbb{R}$ with $y>0$, $\lambda\in\mathbb{C}\setminus\{0\}$, and $s\in\mathbb{C}\setminus\{-1\}$,
\begin{eqnarray}\label{eq2.18}
&&\sum_{k=0}^{n}\chi(k)\lambda^{-k}(k+y)^{s}\nonumber\\
&&\qquad=\sum_{k=2}^{m}\binom{s}{k-1}\frac{(-1)^{k}}{k}\beta_{k,\chi^{-}}(\{n\}_{\chi},\lambda)(n+y)^{s+1-k}\nonumber\\
&&\qquad\quad-\sum_{k=2}^{m}\binom{s}{k-1}\frac{(-1)^{k}}{k}\beta_{k,\chi^{-}}(\{0\}_{\chi},\lambda)y^{s+1-k}\nonumber\\
&&\qquad\quad+\chi(0)y^{s}-\beta_{1,\chi^{-}}(\{n\}_{\chi},\lambda)(n+y)^{s}+\beta_{1,\chi^{-}}(\{0\}_{\chi},\lambda)y^{s}\nonumber\\
&&\qquad\quad+\delta_{1,\lambda^{q}}\beta_{0,\chi^{-}}(\lambda)\biggl(\frac{(n+y)^{s+1}}{s+1}-\frac{y^{s+1}}{s+1}\biggl)\nonumber\\
&&\qquad\quad+(-1)^{m-1}\binom{s}{m}\int_{0}^{n}\beta_{m,\chi^{-}}(\{t\}_{\chi},\lambda)(t+y)^{s-m}dt,
\end{eqnarray}
where the summation on the right hand side of \eqref{eq2.18} vanishes when $m=1$, and $\binom{s}{k}$ are the generalized binomial coefficients defined by
\begin{equation*}
\binom{s}{0}=1,\quad\binom{s}{k}=\frac{s(s-1)\cdots(s-k+1)}{k!}\quad(k\in\mathbb{N}).
\end{equation*}
Let $\overline{\beta}_{n,\chi}(x,\lambda)$ be the $n$-th generalized Apostol-Bernoulli function defined (see, e.g., \cite[Definition 2.7]{he}):
\begin{equation}\label{eq2.19}
\overline{\beta}_{n,\chi}(x,\lambda)=\beta_{n,\chi}(\{x\}_{\chi},\lambda)+\frac{1}{2}\delta_{1,n}\delta_{\mathbb{Z}}(x)\lambda^{-x}\chi(-x),
\end{equation}
where $\beta_{n,\chi}(\{x\}_{\chi},\lambda)=\chi(-1)\beta_{n,\chi^{-}}(\{x\}_{\chi},\lambda)$, $\delta_{\mathbb{Z}}(x)=1$ if $x\in\mathbb{Z}$ and $0$ otherwise. This function is shown in \cite[Theorem 3.1]{he} to have the Fourier series
\begin{equation}\label{eq2.20}
\overline{\beta}_{n,\chi}(x,\lambda)=-\frac{q^{n-1}n!}{\lambda^{x}(2\pi \mathrm{i})^{n}}\sideset{}{'}\sum_{k=-\infty}^{+\infty}\frac{G(k,\chi)e^{\frac{2\pi\mathrm{i}kx}{q}}}{(k-\frac{q\log\lambda}{2\pi\mathrm{i}})^{n}},
\end{equation}
where the dash denotes throughout that undefined terms are excluded from the sum, and $G(k,\chi)$ is the Gauss sum defined by
\begin{equation*}
G(k,\chi)=\sum_{r=1}^{q}\chi(r)e^{\frac{2\pi \mathrm{i}rk}{q}}.
\end{equation*}
Hence, we see from \eqref{eq2.11}, \eqref{eq2.19} and \eqref{eq2.20} that if $|\lambda|=1$ and $\Re(s)<m-1$, then the integral
\begin{equation*}
R(n)=\int_{n}^{\infty}\beta_{m,\chi^{-}}(\{t\}_{\chi},\lambda)(t+y)^{s-m}dt
\end{equation*}
converges absolutely, and $|R(n)|\rightarrow0$ as $n\rightarrow\infty$. So, when $|\lambda|=1$, $\Re(s)<m-1$ and $s\neq-1$, we can write \eqref{eq2.18} as
\begin{eqnarray}\label{eq2.21}
&&\sum_{k=0}^{n}\chi(k)\lambda^{-k}(k+y)^{s}\nonumber\\
&&\qquad=\sum_{k=2}^{m}\binom{s}{k-1}\frac{(-1)^{k}}{k}\beta_{k,\chi^{-}}(\{n\}_{\chi},\lambda)(n+y)^{s+1-k}\nonumber\\
&&\qquad\quad-\sum_{k=2}^{m}\binom{s}{k-1}\frac{(-1)^{k}}{k}\beta_{k,\chi^{-}}(\{0\}_{\chi},\lambda)y^{s+1-k}\nonumber\\
&&\qquad\quad+\chi(0)y^{s}-\beta_{1,\chi^{-}}(\{n\}_{\chi},\lambda)(n+y)^{s}+\beta_{1,\chi^{-}}(\{0\}_{\chi},\lambda)y^{s}\nonumber\\
&&\qquad\quad+\delta_{1,\lambda^{q}}\beta_{0,\chi^{-}}(\lambda)\biggl(\frac{(n+y)^{s+1}}{s+1}-\frac{y^{s+1}}{s+1}\biggl)\nonumber\\
&&\qquad\quad+(-1)^{m-1}\binom{s}{m}\int_{0}^{\infty}\beta_{m,\chi^{-}}(\{t\}_{\chi},\lambda)(t+y)^{s-m}dt\nonumber\\
&&\qquad\quad+(-1)^{m}\binom{s}{m}\int_{n}^{\infty}\beta_{m,\chi^{-}}(\{t\}_{\chi},\lambda)(t+y)^{s-m}dt.
\end{eqnarray}
In particular, for $\Re(s)<-1$, letting $n\rightarrow\infty$ in \eqref{eq2.21}, we obtain
\begin{eqnarray}\label{eq2.22}
&&\sum_{k=0}^{\infty}\frac{\chi(k)\lambda^{-k}}{(k+y)^{-s}}\nonumber\\
&&\qquad=-\sum_{k=2}^{m}\binom{s}{k-1}\frac{(-1)^{k}}{k}\beta_{k,\chi^{-}}(\{0\}_{\chi},\lambda)y^{s+1-k}\nonumber\\
&&\qquad\quad+\chi(0)y^{s}+\beta_{1,\chi^{-}}(\{0\}_{\chi},\lambda)y^{s}-\delta_{1,\lambda^{q}}\beta_{0,\chi^{-}}(\lambda)\frac{y^{s+1}}{s+1}\nonumber\\
&&\qquad\quad+(-1)^{m-1}\binom{s}{m}\int_{0}^{\infty}\beta_{m,\chi^{-}}(\{t\}_{\chi},\lambda)(t+y)^{s-m}dt.
\end{eqnarray}
Inserting \eqref{eq2.22} into \eqref{eq2.21}, we know that for $|\lambda|=1$,
\begin{eqnarray}\label{eq2.23}
&&\sum_{k=0}^{n}\chi(k)\lambda^{-k}(k+y)^{s}\nonumber\\
&&\qquad=\sum_{k=0}^{\infty}\frac{\chi(k)\lambda^{-k}}{(k+y)^{-s}}+\sum_{k=2}^{m}\binom{s}{k-1}\frac{(-1)^{k}}{k}\beta_{k,\chi^{-}}(\{n\}_{\chi},\lambda)(n+y)^{s+1-k}\nonumber\\
&&\qquad\quad-\beta_{1,\chi^{-}}(\{n\}_{\chi},\lambda)(n+y)^{s}+\delta_{1,\lambda^{q}}\beta_{0,\chi^{-}}(\lambda)\frac{(n+y)^{s+1}}{s+1}\nonumber\\
&&\qquad\quad+(-1)^{m}\binom{s}{m}\int_{n}^{\infty}\beta_{m,\chi^{-}}(\{t\}_{\chi},\lambda)(t+y)^{s-m}dt,
\end{eqnarray}
which holds for $\Re(s)\leqslant m-1$ and $s\neq-1$ by analytic continuation. Now, taking $s=m-1$ and $n=0$ in \eqref{eq2.23}, we have
\begin{eqnarray}\label{eq2.24}
&&\sum_{k=0}^{\infty}\frac{\chi(k)\lambda^{-k}}{(k+y)^{1-m}}\nonumber\\
&&\qquad=-\frac{1}{m}\sum_{k=0}^{m}\binom{m}{k}(-1)^{k}\beta_{k,\chi^{-}}(\{0\}_{\chi},\lambda)y^{m-k}+\chi(0)y^{m-1}.
\end{eqnarray}
Note that from \eqref{eq2.2} and \eqref{eq2.17} we have
\begin{equation}\label{eq2.25}
\beta_{n,\chi}(x,\lambda)=q^{n-1}\sum_{r=1}^{q}\chi(r)\lambda^{r}\beta_{n}\biggl(\frac{x+r}{q},\lambda^{q}\biggl)\quad(n\in\mathbb{N}_{0}).
\end{equation}
Moreover, for $n\in\mathbb{N}$ and $x\in\mathbb{R}$ (see, e.g., \cite[Propostion 2.8]{he}),
\begin{equation}\label{eq2.26}
\beta_{n,\chi}(\{x\}_{\chi},\lambda)=\beta_{n,\chi}(x,\lambda)-n\underset{\substack{0\leqslant r\leqslant x\\ r\in\mathbb{Z}}}{\sum}\chi(-r)\lambda^{-r}(x-r)^{n-1}.
\end{equation}
From \eqref{eq2.10}, \eqref{eq2.25} and \eqref{eq2.26}, it follows that for $n\in\mathbb{N}_{0}$,
\begin{equation}\label{eq2.27}
\beta_{n,\chi^{-}}(\{0\}_{\chi},\lambda)=\beta_{n,\chi^{-}}(0,\lambda)-\delta_{1,n}\chi(0).
\end{equation}
Applying \eqref{eq2.27} to the right-hand side of \eqref{eq2.24}, we have
\begin{equation}\label{eq2.28}
\sum_{k=0}^{\infty}\frac{\chi(k)\lambda^{-k}}{(k+y)^{1-m}}=-\frac{1}{m}\sum_{k=0}^{m}\binom{m}{k}(-1)^{k}\beta_{k,\chi^{-}}(0,\lambda)y^{m-k}.
\end{equation}
It is shown in \cite[Proposition 2.3]{he} that if $q=1$, then
\begin{equation}\label{eq2.29}
\beta_{n,\chi}(-x,\lambda)=(-1)^{n}\beta_{n}(x,\lambda^{-1})\quad(n\in\mathbb{N}_{0}),
\end{equation}
and if $q\geqslant2$, then
\begin{equation}\label{eq2.30}
\beta_{n,\chi}(-x,\lambda)=(-1)^{n}\beta_{n,\chi^{-}}(x,\lambda^{-1})\quad(n\in\mathbb{N}_{0}).
\end{equation}
Thus, replacing $\lambda$ by $\lambda^{-1}$ in \eqref{eq2.28}, with the help of \eqref{eq2.29} and \eqref{eq2.30}, we find from \eqref{eq2.2} and \eqref{eq2.17} that if $q=1$, then
\begin{eqnarray*}
\sum_{k=0}^{\infty}\frac{\chi(k)\lambda^{k}}{(k+y)^{1-m}}&=&-\frac{1}{m}\sum_{k=0}^{m}\binom{m}{k}\beta_{k}(0,\lambda)y^{m-k}\nonumber\\
&=&-\frac{\beta_{m}(y,\lambda)}{m},
\end{eqnarray*}
and if $q\geqslant2$, then
\begin{eqnarray*}
\sum_{k=0}^{\infty}\frac{\chi(k)\lambda^{k}}{(k+y)^{1-m}}&=&-\frac{1}{m}\sum_{k=0}^{m}\binom{m}{k}\beta_{k,\chi}(0,\lambda)y^{m-k}\nonumber\\
&=&-\frac{\beta_{m,\chi}(y,\lambda)}{m}.
\end{eqnarray*}
Now the desired result follows by taking $\lambda=e^{\frac{2\pi \mathrm{i}x}{q}}$ and replacing $y$ by $a$ and $m$ by $n$ in the above two identities. This completes the proof of Theorem \ref{thm2.6}.
\end{proof}

We note that formula \eqref{eq2.15} is due to Apostol \cite{apostol2}, who proved it using contour integration, while formula \eqref{eq2.16} improves upon the result presented in \cite[Corollary 3.7]{he}. In particular, if we take $x=0$ in \eqref{eq2.15}, then we know that for $n\in\mathbb{N}$ and $a\in\mathbb{R}$ with $a>0$,
\begin{equation}\label{eq2.31}
\zeta(1-n,a)=-\frac{B_{n}(a)}{n},
\end{equation}
where $\zeta(s,a)$ is the Hurwitz zeta-function given for $s\in\mathbb{C}$, $a\in\mathbb{R}$ with $a>0$ by
\begin{equation*}
\zeta(s,a)=\sum_{n=0}^{\infty}\frac{1}{(n+a)^{s}}\quad(\Re(s)>1).
\end{equation*}
In addition, if we take $x=1/2$ in \eqref{eq2.15}, then we find that for $n\in\mathbb{N}$ and $a\in\mathbb{R}$ with $a>0$,
\begin{equation}\label{eq2.32}
\eta(1-n,a)=\frac{1}{2}E_{n-1}(a),
\end{equation}
where $\eta(s,a)$ is the alternating Hurwitz-zeta function given for $s\in\mathbb{C}$, $a\in\mathbb{R}$ with $a>0$ by
\begin{equation*}
\eta(s,a)=\sum_{n=0}^{\infty}(-1)^{n}\frac{1}{(n+a)^{s}}\quad(\Re(s)>0).
\end{equation*}

\section{The proof of Theorem \ref{thm2.1}}\label{sec3}

We first prove that the case $a\geq0$ in \eqref{eq2.1} holds. Let
\begin{equation}\label{eq3.1}
R_{m+1}(\chi^{-},f)=\frac{(-1)^{m}}{(m+1)!}\int_{a}^{b}\beta_{m+1,\chi^{-}}(\{t\}_{\chi},\lambda)f^{(m+1)}(t)dt.
\end{equation}
By taking the derivative with respect to $x$ on both sides of \eqref{eq2.20}, we obtain that for $n\geqslant2$ and $x\in\mathbb{R}$,
\begin{equation}\label{eq3.2}
\frac{\partial}{\partial x}\overline{\beta}_{n,\chi}(x,\lambda)=n\overline{\beta}_{n-1,\chi}(x,\lambda).
\end{equation}
It follows from \eqref{eq2.19} and \eqref{eq3.2} that
\begin{equation}\label{eq3.3}
\frac{\partial}{\partial x}\beta_{n,\chi^{-}}(\{x\}_{\chi},\lambda)=n\beta_{n-1,\chi^{-}}(\{x\}_{\chi},\lambda)
\end{equation}
for $n\geqslant2$ and $x\in\mathbb{R}$, except when $n=2$ and $x$ belongs to a certain set of measure zero (namely, the set of $x$ for which there exists $k\in\{1,2,\ldots,q\}$ such that $q\mid(x+k)$). Hence, using integration by parts in \eqref{eq3.1}, we know from \eqref{eq3.3} that
\begin{eqnarray}\label{eq3.4}
&&R_{m+1}(\chi^{-},f)\nonumber\\
&&\qquad=\frac{(-1)^{m}}{(m+1)!}\bigl(\beta_{m+1,\chi^{-}}(\{b\}_{\chi},\lambda)f^{(m)}(b)-\beta_{m+1,\chi^{-}}(\{a\}_{\chi},\lambda)f^{(m)}(a)\bigl)\nonumber\\
&&\qquad\quad+R_{m}(\chi^{-},f)\nonumber\\
&&\qquad=\sum_{k=1}^{m}\frac{(-1)^{k}}{(k+1)!}\bigl(\beta_{k+1,\chi^{-}}(\{b\}_{\chi},\lambda)f^{(k)}(b)-\beta_{k+1,\chi^{-}}(\{a\}_{\chi},\lambda)f^{(k)}(a)\bigl)\nonumber\\
&&\qquad\quad+R_{1}(\chi^{-},f).
\end{eqnarray}
Now, taking $n=1$ in \eqref{eq2.26}, we find from \eqref{eq2.10} and \eqref{eq2.11} that if $\lambda^{q}=1$, then for $t\in\mathbb{R}$,
\begin{eqnarray}\label{eq3.5}
\beta_{1,\chi^{-}}(\{t\}_{\chi},\lambda)&=&\frac{t}{q}\sum_{r=1}^{q}\chi(-r)\lambda^{r}+\sum_{r=1}^{q}\chi(-r)\lambda^{r}\biggl(\frac{r}{q}-\frac{1}{2}\biggl)-\underset{\substack{0\leqslant r\leqslant t\\ r\in\mathbb{Z}}}{\sum}\chi(r)\lambda^{-r}\nonumber\\
&=&t\beta_{0,\chi^{-}}(\lambda)+\beta_{1,\chi^{-}}(\lambda)-\underset{\substack{0\leqslant r\leqslant t\\ r\in\mathbb{Z}}}{\sum}\chi(r)\lambda^{-r},
\end{eqnarray}
and if $\lambda^{q}\neq1$, then for $t\in\mathbb{R}$,
\begin{equation}\label{eq3.6}
\beta_{1,\chi^{-}}(\{t\}_{\chi},\lambda)=\frac{1}{\lambda^{q}-1}\sum_{r=1}^{q}\chi(-r)\lambda^{r}-\underset{\substack{0\leqslant r\leqslant t\\ r\in\mathbb{Z}}}{\sum}\chi(r)\lambda^{-r}.
\end{equation}
Clearly, from \eqref{eq3.5} and by using integration by parts, we obtain that for $\lambda^{q}=1$,
\begin{eqnarray}\label{eq3.7}
R_{1}(\chi^{-},f)&=&\int_{a}^{b}\bigl(t\beta_{0,\chi^{-}}(\lambda)+\beta_{1,\chi^{-}}(\lambda)-\underset{\substack{0\leqslant r\leqslant t\\ r\in\mathbb{Z}}}{\sum}\chi(r)\lambda^{-r}\bigl)f^{(1)}(t)dt\nonumber\\
&=&\bigl(b\beta_{0,\chi^{-}}(\lambda)+\beta_{1,\chi^{-}}(\lambda)\bigl)f(b)-\bigl(a\beta_{0,\chi^{-}}(\lambda)+\beta_{1,\chi^{-}}(\lambda)\bigl)f(a)\nonumber\\
&&-\beta_{0,\chi^{-}}(\lambda)\int_{a}^{b}f(t)dt-\int_{a}^{b}\underset{\substack{0\leqslant r\leqslant t\\ r\in\mathbb{Z}}}{\sum}\chi(r)\lambda^{-r}f^{(1)}(t)dt.
\end{eqnarray}
Observe that
\begin{eqnarray}\label{eq3.8}
&&\int_{a}^{b}\underset{\substack{0\leqslant r\leqslant t\\ r\in\mathbb{Z}}}{\sum}\chi(r)\lambda^{-r}f^{(1)}(t)dt\nonumber\\
&&\qquad=\int_{a}^{[a]+1}\underset{\substack{0\leqslant r\leqslant t\\ r\in\mathbb{Z}}}{\sum}\chi(r)\lambda^{-r}f^{(1)}(t)dt+\int_{[a]+1}^{[a]+2}\underset{\substack{0\leqslant r\leqslant t\\ r\in\mathbb{Z}}}{\sum}\chi(r)\lambda^{-r}f^{(1)}(t)dt\nonumber\\
&&\qquad\quad+\cdots+\int_{[b]}^{b}\underset{\substack{0\leqslant r\leqslant t\\ r\in\mathbb{Z}}}{\sum}\chi(r)\lambda^{-r}f^{(1)}(t)dt\nonumber\\
&&\qquad=\underset{\substack{0\leqslant r\leqslant a\\ r\in\mathbb{Z}}}{\sum}\chi(r)\lambda^{-r}\int_{a}^{b}f^{(1)}(t)dt\nonumber\\
&&\qquad\quad+\chi([a]+1)\lambda^{-([a]+1)}\int_{[a]+1}^{[a]+2}f^{(1)}(t)dt\nonumber\\
&&\qquad\quad+\cdots\nonumber\\
&&\qquad\quad+\bigl(\chi([a]+1)\lambda^{-([a]+1)}+\cdots+\chi([b])\lambda^{-[b]}\bigl)\int_{[b]}^{b}f^{(1)}(t)dt\nonumber\\
&&\qquad=\underset{\substack{0\leqslant r\leqslant a\\ r\in\mathbb{Z}}}{\sum}\chi(r)\lambda^{-r}\int_{a}^{b}f^{(1)}(t)dt+\underset{\substack{a<r\leqslant b\\ r\in\mathbb{Z}}}{\sum}\chi(r)\lambda^{-r}\int_{r}^{b}f^{(1)}(t)dt.
\end{eqnarray}
From \eqref{eq3.8}, it follows that
\begin{eqnarray}\label{eq3.9}
&&\int_{a}^{b}\underset{\substack{0\leqslant r\leqslant t\\ r\in\mathbb{Z}}}{\sum}\chi(r)\lambda^{-r}f^{(1)}(t)dt\nonumber\\
&&\qquad=f(b)\underset{\substack{0\leqslant r\leqslant b\\ r\in\mathbb{Z}}}{\sum}\chi(r)\lambda^{-r}-f(a)\underset{\substack{0\leqslant r\leqslant a\\ r\in\mathbb{Z}}}{\sum}\chi(r)\lambda^{-r}-\underset{\substack{a<r\leqslant b\\ r\in\mathbb{Z}}}{\sum}\chi(r)\lambda^{-r}f(r).
\end{eqnarray}
Inserting \eqref{eq3.9} into \eqref{eq3.7}, with the help of \eqref{eq3.5}, we know that for $\lambda^{q}=1$,
\begin{eqnarray}\label{eq3.10}
R_{1}(\chi^{-},f)&=&\beta_{1,\chi^{-}}(\{b\}_{\chi},\lambda)f(b)-\beta_{1,\chi^{-}}(\{a\}_{\chi},\lambda)f(a)\nonumber\\
&&-\beta_{0,\chi^{-}}(\lambda)\int_{a}^{b}f(t)dt+\underset{\substack{a< r\leqslant b\\ r\in\mathbb{Z}}}{\sum}\chi(r)\lambda^{-r}f(r).
\end{eqnarray}
Hence, we see from \eqref{eq3.4} and \eqref{eq3.10} that the condition $\lambda^{q}=1$ in \eqref{eq2.1} holds. Similarly, we obtain from \eqref{eq3.6} and \eqref{eq3.9} that for $\lambda^{q}\neq1$,
\begin{eqnarray}\label{eq3.11}
R_{1}(\chi^{-},f)
&=&\beta_{1,\chi^{-}}(\lambda)f(b)-\beta_{1,\chi^{-}}(\lambda)f(a)-\int_{a}^{b}\underset{\substack{0\leqslant r\leqslant t\\ r\in\mathbb{Z}}}{\sum}\chi(r)\lambda^{-r}f^{(1)}(t)dt\nonumber\\
&=&\beta_{1,\chi^{-}}(\{b\}_{\chi},\lambda)f(b)-\beta_{1,\chi^{-}}(\{a\}_{\chi},\lambda)f(a)\nonumber\\
&&+\underset{\substack{a< r\leqslant b\\ r\in\mathbb{Z}}}{\sum}\chi(r)\lambda^{-r}f(r).
\end{eqnarray}
Putting \eqref{eq3.11} into \eqref{eq3.4}, we see that the condition $\lambda^{q}\neq1$ in \eqref{eq2.1} holds. Therefore, \eqref{eq2.1} holds for the case $a\geqslant0$.

For a general $a$, we choose suitable $n\in\mathbb{Z}$ such that $a_{1}=a+nq\geqslant0$, and then set $b_{1}=b+nq$ and $f_{1}(x)=f(x-nq)$. Obviously, $f_{1}(x)$ is an $m+1$
times continuously differentiable function on the interval $[a_{1},b_{1}]$. Hence, from the above, we have
\begin{eqnarray}\label{eq3.12}
&&\sum_{\substack{a_{1}<k\leqslant b_{1}\\k\in\mathbb{Z}}}\chi(k)\lambda^{-k}f_{1}(k)\nonumber\\
&&\qquad=\sum_{k=0}^{m}\frac{(-1)^{k+1}}{(k+1)!}\bigl(\beta_{k+1,\chi^{-}}(\{b_{1}\}_{\chi},\lambda)f_{1}^{(k)}(b_{1})-\beta_{k+1,\chi^{-}}(\{a_{1}\}_{\chi},\lambda)f_{1}^{(k)}(a_{1})\bigl)\nonumber\\
&&\qquad\quad+\delta_{1,\lambda^{q}}\beta_{0,\chi^{-}}(\lambda)\int_{a_{1}}^{b_{1}}f_{1}(t)dt\nonumber\\
&&\qquad\quad+\frac{(-1)^{m}}{(m+1)!}\int_{a_{1}}^{b_{1}}\beta_{m+1,\chi^{-}}(\{t\}_{\chi},\lambda)f_{1}^{(m+1)}(t)dt.
\end{eqnarray}
Since $\chi(k+nq)=\chi(k)$ for $k\in\mathbb{Z}$ and $\beta_{n,\chi^{-}}(\{x+nq\}_{\chi},\lambda)=\lambda^{-nq}\beta_{n,\chi^{-}}(\{x\}_{\chi},\lambda)$ for $x\in\mathbb{R}$, we see that \eqref{eq3.12} can be rewritten as
\begin{eqnarray*}
&&\sum_{\substack{a<k\leqslant b\\k\in\mathbb{Z}}}\chi(k)\lambda^{-k}f(k)\nonumber\\
&&\qquad=\sum_{k=0}^{m}\frac{(-1)^{k+1}}{(k+1)!}\bigl(\beta_{k+1,\chi^{-}}(\{b\}_{\chi^{-}},\lambda)f^{(k)}(b)-\beta_{k+1,\chi^{-}}(\{a\}_{\chi^{-}},\lambda)f^{(k)}(a)\bigl)\nonumber\\
&&\qquad\quad+\delta_{1,\lambda^{q}}\beta_{0,\chi^{-}}(\lambda)\int_{a}^{b}f(t)dt\nonumber\\
&&\qquad\quad+\frac{(-1)^{m}}{(m+1)!}\int_{a}^{b}\beta_{m+1,\chi^{-}}(\{t\}_{\chi^{-}},\lambda)f^{(m+1)}(t)dt,
\end{eqnarray*}
as desired. This finishes the proof of Theorem \ref{thm2.1}.

\end{document}